\documentclass[12 pt]{amsart}
\usepackage{amssymb,latexsym,amsmath,amscd,amsthm,graphicx, color}
\usepackage[all]{xy}
\usepackage{pgf,tikz}
\usepackage{mathrsfs}
\usepackage{cite}
\theoremstyle{plain}

\newtheorem{theorem}[subsection]{Theorem}

\newtheorem{lemma}[subsection]{Lemma}

\newtheorem{prop}[subsection]{Proposition}
\theoremstyle{definition}

\newtheorem{remark}[subsection]{Remark}

\newtheorem{note}[subsection]{Note}

\newcommand{\uu}{\cup}
\newcommand{\ii}{\cap}

\newcommand{\set}[1]{\{#1\}}

\newcommand{\ga}{\alpha}
\newcommand{\gb}{\beta}

\newcommand{\gd}{\delta}

\newcommand{\gl}{\lambda}

\newcommand{\gq}{\theta}
\newcommand{\gs}{\sigma}

\newcommand{\tit}{\textit}

\newcommand{\C}[1]{\mathcal{#1}}
\newcommand{\D}[1]{\mathbb{#1}}
\providecommand{\R}[0]{\mathbb{R}}

\newcommand{\te}{\text}

\newcommand{\ds}{\displaystyle}
\newcommand{\diff}[2]{\frac{\partial #1}{\partial #2}}

\begin{document}

\title{High precision numerical computation of principal points for univariate distributions}

\author{Santanu Chakraborty}
\address{School of Mathematical and Statistical Sciences\\
University of Texas Rio Grande Valley\\
1201 West University Drive\\
Edinburg, TX 78539-2999, USA.}
\email{santanu.chakraborty@utrgv.edu}

\author{ Mrinal Kanti Roychowdhury}
\address{School of Mathematical and Statistical Sciences\\
University of Texas Rio Grande Valley\\
1201 West University Drive\\
Edinburg, TX 78539-2999, USA.}
\email{mrinal.roychowdhury@utrgv.edu}

\author{Josef Sifuentes}
\address{School of Mathematical and Statistical Sciences\\
University of Texas Rio Grande Valley\\
1201 West University Drive\\
Edinburg, TX 78539-2999, USA.}
\email{josef.sifuentes@utrgv.edu}


\thanks{}

\date{}
\maketitle

\pagestyle{myheadings}\markboth{S. Chakraborty, M.K. Roychowdhury, and J. Sifuentes}{Computation of principal points for univariate distributions}

\begin{abstract}
Principal points were first introduced by Flury: for a positive integer $n$, $n$ principal points of a random variable are the $n$ points that minimize the mean squared distance between the random variable and the nearest of the $n$ points. In this paper, we determine the $n$ principal points and the corresponding values of mean squared distance for different values of $n$ for some univariate absolutely continuous distributions.
\end{abstract}

\section{Introduction}

Quantization is a process of approximation with broad applications in signal processing and data compression (see \cite{DFG, GG, GKL, GN, Z1}). For rigorous mathematical treatment of the quantization theory one can see Graf-Luschgy's book (see \cite{GL1}). Quantization for probability distributions concerns the best approximation of a probability measure $P$ defined on a metric space by a measure supported on a finite number of points, or in other words, the best approximation of a $d$-dimensional random vector $X$ with distribution $P$ by a random vector $Y$ with at most $n$-values in its image. Let $\D R^d$ denote the $d$-dimensional Euclidean space equipped with the Euclidean norm $\|\cdot\|$, and let $P$ be a Borel probability measure on $\D R^d$. Let $0<r<+\infty$.  Then, the $n$th \textit{quantization
error} for $P$ of order $r$ is defined by
\begin{equation*} \label{eq1} V_{n, r}:=V_{n, r}(P)=\inf \Big\{\int \min_{a\in\alpha} \|x-a\|^r dP(x) : \ga \subset \mathbb R^d, \text{ card}(\ga) \leq n \Big\},\end{equation*}
where the infimum is taken over all subsets $\alpha$ of $\mathbb R^d$ with $\te{card}(\alpha)\leq n$. We assume that the probability distribution $P$ has finite second moment, i.e., $\int \| x\|^r dP(x)<\infty$. Then, there is some set $\alpha$ for
which the infimum is achieved (see \cite{GL1}). A set $\alpha$ for which the infimum is achieved, i.e.,
$V_{n, r}=\int \min_{a \in \ga} \|x-a\|^r dP(x),$ is called an \textit{optimal set of $n$-means}, or \tit{optimal set of $n$-quantizers} (of order $r$). Elements of an optimal set of $n$-means are refereed to as \tit{optimal centers}, or \tit{optimal quantizers}. The collection of all optimal sets of $n$-means for a probability measure $P$ is denoted by $\C C_n:=\C C_n(P)$. We assume that $P$ is continuous. Then, an optimal set of $n$-means always has exactly $n$ elements (see \cite{GL1}). Throughout the paper, we will keep $r=2$, and will denote the $n$th quantization error of order $2$ by $V_n:=V_{n,2}(P)$. If $P$ is continuous with finite second moment, and $r=2$, then the elements in an optimal set of $n$-means are also referred to as \tit{principal points}. In other words, principal points are defined as the set of $n$ points that minimizes the expected value of the squared distance between the random variable $X$ with distribution $P$ and the nearest of the points in the set (see \cite{F1}). 

Since the introduction of principal points by Flury (see \cite{F1}), there have been almost three decades of research on principal points. Flury considered $k$ principal points for a $p$-variate random vector, principal points for  univariate symmetric distributions, univariate and bivariate normal distributions, multivariate elliptical distributionsin the very first paper (see \cite{F1}). In his 1993 paper (see \cite{F2}), Flury redefined principal points in terms of self-consistent points (which we define below) and also described four methods of estimation of principal points which include maximum likelihood estimation, k-means algorithm etc. This study was carried forward by a number of papers in the 90s (see \cite{T1, T2, T3, T4, T5, T6, TF, TLF, Z2, Z3}) which dealt with principal points of strongly unimodal distributions, strongly symmetric multivariate distributions, uniqueness of $k$ principal points for univariate normal distributions, principal points and self-consistent points of $p$-variate elliptical distributions, method of determining principal points for univariate continuous distributions, uniqueness and symmetry of self-consistent points for univariate continuous distributions etc. Tarpey, Li, Zopp\`{e} and Flury himself were the main contributers. Then, in the new century, there have been more attempts of estimating principal points of which the first one that looks interesting is by Stamfer and Stanlober (see \cite{SS}). Tarpey remained active even in this century as he collaborated with Matsuura and Kurata (see \cite{MKT}). Matsuura and Kurata have several results on principal points for mixture distributions (see \cite{MK1, MK2, MK4}) and they also introduced $m$-dimensional $n$ principal points (see \cite{MK3}). In fact, principal points of univariate and multivariate location mixtures were first studied by Yamamoto and Shinozaki (see \cite{YS1, YS2}) and extended by Kurata (see \cite{K2}). In the recent years, Yamashita and Suzuki (see \cite{YS3, YS4}) have been very active on principal points related to binary distributions and they have also collaborated with Matsuura in this regard (see \cite{YMS1, YMS2}). 

Our goal in this paper is very simple. We aim to calculate $n$ principal points for univariate continuous distributions by a high precision algorithmic approach. To this end, we make some definition and notations. If $\ga$ is a set of $n$ principal points, then we call it an $n$-principal set. Thus,
\[V_n:=V_n(P)=E\Big(\min_{a\in \ga} \|X-a\|^2\Big)=\int\min_{a\in \ga}\|x-a\|^2dP(x),\]
where, by $E(X)$, it is meant the expected value of the random variable $X$. For a finite set $\ga$, the error $\int \min_{a \in \ga} \|x-a\|^2 dP(x)$ is often referred to as the \tit{cost} or \tit{distortion error} or \tit{$n$-th mean squared distance} for $\ga$, and is denoted by $V(P; \ga)$. Thus, if $\ga$ is a set of $n$ principal points, then $V_n:=V_n(P)=V(P; \ga)$. For a finite subset $\ga$ of $\D R^d$, the \tit{Voronoi region} generated by an element $a\in \ga$ is the set of all elements in $\D R^d$ that have $a$ as their nearest point in $\ga$, and is denoted by $M(a|\ga)$, i.e.,
\[M(a|\ga)=\set{x \in \D R^d : \|x-a\|=\min_{b \in \ga}\|x-b\|}.\]
If $P$ is a continuous probability measure, then the set of all boundary points of the Voronoi regions has probability measure zero, i.e., $P(\partial M(a|\ga))=0$ for all $a\in \ga$.
The set $\ga$ with respect to the probability distribution $P$ is called \tit{self-consistent} if for each $a\in \ga$,
\begin{equation*} \label{eq2}
a=E(X : X \in M(a|\ga)),
\end{equation*}
i.e., if each $a\in \ga$ equals the conditional expectation of the random variable $X$ given that $X$ is closest to $a$. Flury showed that principal points are self-consistent (see \cite{F2}), but that the converse is not always true (for example, see \cite{DR, R2}). For a given value of $n$, a distribution can have several different sets of $n$ self-consistent points, for example, for an absolutely continuous probability measure see \cite{T4}, and for a singular continuous probability measure see \cite{R1}. It is also possible to have more than one set of $n$ principal points, for example, for an absolutely continuous probability measure see \cite{DR, R2}, and for a singular continuous probability measure see \cite{GL2}. Notice that for a given $n\in \D N$, if several sets of $n$ self-consistent points exist, the self-consistent set(s) with smallest distortion error(s) will give the optimal set(s) of $n$-means.
Finding an optimal set of $n$-means for a univariate distribution is often a straight forward numerical problem if there exists a unique set of $n$ self-consistent points. The problem of finding an optimal set of $n$-means for a multivariate distribution is considerably much more difficult as the paper \cite{DR} illustrates. Optimal sets of $n$-means and the $n$th quantization errors for $n=1, 2, \cdots, 5$ for standard normal distribution were first calculated by Flury (see \cite[Table~1]{F1}). For $n=1, 2, \cdots, 8$, they were calculated by Graf-Luschgy (see \cite[Table~5.1]{GL1}). Later, for $n=1, 2, \cdots, 10$, they were calculated by Matsuura et al. (see \cite[Table~1]{MKT}). Optimal sets of $n$-means and the $n$th quantization errors for standard exponential distribution for $n=1, 2, \cdots, 6$, were calculated by Zopp\`{e} (see \cite[Table~4.2]{Z2}), and for $n=1, 2, \cdots, 8$, they were calculated by Graf-Luschgy (see \cite[Table~5.4]{GL1}).


In this paper, we present a high precision algorithm for calculating $n$ principal points and the $n$th quantization errors. We do this by applying Newton's algorithm to nonlinear equations and using adaptive quadrature routines built in to Matlab \cite{ML} whenever a probability distribution function does not allow for explicit integration. Furthermore, we formulate the linearization of the principal point nonlinear equation, which is solved at each iteration of Newton's algorithm.

\section{ preliminaries}

Throughout the paper we will assume that $P$ is a continuous Borel probability measure with univariate density function $f$.
A probability measure $P$ with density function $f$ is called \tit{symmetric} about $0$ if $f(-x)=f(x)$, i.e., if the pdf is an even function. On the other hand, a nonempty subset $\ga$ of $\D R$ is called \tit{symmetric} if $\ga=-\ga$. A non-negative function $f : \D R^d \to \D R^+$ is \tit{logarithmically concave}, or \tit{logconcave} for short, if its domain is a convex set, and if it satisfies $\log f(\gq x + ( 1 - \gq ) y)\geq \gq\log  f ( x ) +(1-\gq)\log f( y )$ for all $x, y \in \te{dom}f$  and $0 < \gq < 1$. Thus, we see that uniform distribution, normal distribution, double exponential distribution, and exponential distribution, in fact all the distributions considered in this paper, are all logconcave functions. The probability measure $P$ is called \tit{strongly unimodal} if $P=f\gl$, where $\gl$ is the Lebesgue measure, such that $I:=\set {h>0}$ is an open (possibly unbounded) interval and $\log f$ is concave on $I$.

The following theorem is known.

\begin{theorem} (see \cite[Theorem~5.1]{GL1}) \label{Th0}
Suppose that $P$ is strongly unimodal. Then, for every  $n\in \D N$, the $n$-principal set for $P$ is unique.
\end{theorem}

In the following note we give a method of obtaining an $n$-principal set for a univariate absolutely continuous probability measure with density function $f(x)$.
\begin{note}\label{note22}
Suppose that $\mathcal D$ is the domain of the probability density function $f(x)$, with limiting values
$$
c := \inf(\mathcal D),
\qquad
d := \sup(\mathcal D),
$$
and
let $\ga_n:=\set{a_1, a_2, \cdots, a_n}$ be an $n$-principal set for $P$ with probability density function $f(x)$ such that $-\infty<a_1<a_2<\cdots<a_n<\infty$. Write
\begin{align*}\label{eq234}
M(a_i|\ga_n):=\left\{\begin{array}{cc}
\left(c, \frac{a_1+a_2}{2}\right] & \te{ if } i=1, \\[1em]
\left[\frac{a_{i-1}+a_i}{2}, \frac{a_i+a_{i+1}}{2}\right] & \te{ if } 2 \leq i \leq n-1, \\[1em]
\left[\frac{a_{n-1}+a_n}{2}, d \right) & \te{ if } i=n,
\end{array}
\right.
\end{align*}
where $M(a_i|\ga_n)$ represent the Voronoi regions of $a_i$ for all $1\leq i\leq n$ with respect to the set $\ga_n$.
Since the principal points are the expected values of their own Voronoi regions, we have
\begin{equation} \label{eq00} a_i=E(X : X \in M(a_i|\ga_n))
 \end{equation}
 for all $1\leq i\leq n$.
 Solving the $n$ equations we can obtain the $n$-principal sets for $P$. Once, an $n$-principal set is known, the corresponding $n$th mean squared distance can easily be determined.
\end{note}
The following lemma will be convenient.
\begin{lemma} \label{prop0} Let $\set{a_1<a_2<\cdots<a_n}$ be an $n$-principal set for a continuous univariate probability distribution with density function $f(x)$, and $V_n(f(x))$ be the corresponding $n$-the mean squared distance. Let $\mu$ and $\gs$ be two arbitrary constants. Set $b_i:=\mu +a_i\gs$ for $1\leq i\leq n$. Then, $\set{b_1<b_2< \cdots< b_n}$ is an $n$-principal set for the univariate continuous probability measure with density function $f(\frac{x-\mu}{\gs})$, and the corresponding mean squared distance is $V_n(f(\frac{x-\mu}{\gs}))$ given by $V_n(f(\frac{x-\mu}{\gs}))=\gs^2 V_n(f(x))$.
\end{lemma}

\begin{proof} Let $\set{a_1<a_2<\cdots<a_n}$ be an optimal set of $n$-means for the univariate continuous probability measure with density function $f(x)$. Write $a_0=c$ and $a_{n+1}= d$. Let $b_i=\mu +a_i\gs$ for $0\leq i\leq (n+1)$. Then, for $1\leq i\leq n$, by Note~\ref{note22}, we have
\begin{align*}
a_i=\frac {\int_{\frac{a_{i-1}+a_i}{2}}^{\frac{a_{i}+a_{i+1}}{2}} x f(x)dx}{\int_{\frac{a_{i-1}+a_i}{2}}^{\frac{a_{i}+a_{i+1}}{2}}  f(x) dx}=\frac {\int_{\frac{b_{i-1}+b_i}{2}}^{\frac{b_{i}+b_{i+1}}{2}} (\frac {x-\mu}{\gs}) f(\frac {x-\mu}{\gs}) dx}{\int_{\frac{b_{i-1}+b_i}{2}}^{\frac{b_{i}+b_{i+1}}{2}}   f(\frac {x-\mu}{\gs})dx} \te{ implying }  b_i=\frac {\int_{\frac{b_{i-1}+b_i}{2}}^{\frac{b_{i}+b_{i+1}}{2}} x f(\frac {x-\mu}{\gs})dx}{\int_{\frac{b_{i-1}+b_i}{2}}^{\frac{b_{i}+b_{i+1}}{2}}   f(\frac {x-\mu}{\gs})dx},
\end{align*}
i.e., $b_i=E(X : X\in M(b_i|\set{b_1, b_2, \cdots, b_n})$, where $X$ is a random variable with probability density function $f(\frac {x-\mu}{\gs})$. Thus, we see that $\set{b_1<b_2< \cdots< b_n}$  forms an optimal set of $n$-means for the univariate continuous probability measure with density function $f(\frac{x-\mu}{\gs})$. Moreover, by the definition of $n$th quantization error, we have
\begin{align*} V_n(f(x))=\sum_{i=1}^n \int_{\frac{a_{i-1}+a_i}{2}}^{\frac{a_{i}+a_{i+1}}{2}} (x-a_i)^2 f(x)dx=\frac 1{\gs^2} \sum_{i=1}^n\int_{\frac{b_{i-1}+b_i}{2}}^{\frac{b_{i}+b_{i+1}}{2}}(x-b_i)^2  f(\frac {x-\mu}{\gs})dx,\end{align*}
yielding $V_n(f(\frac{x-\mu}{\gs}))=\gs^2 V_n(f(x))$, which is the proposition.
\end{proof}

\begin{remark} By Lemma~\ref{prop0}, it is clear that for any positive integer $n$, to determine an $n$-principal set for any normal distribution or any exponential distribution, it is enough to determine an $n$-principal set for the standard normal distribution with density function $f(x)=\frac 1{\sqrt{2\pi}} e^{-\frac{x^2}{2}}$ for $-\infty<x<\infty$, or standard exponential distribution with density function $f(x)=e^{-x}$ for $x\geq 0$.
\end{remark}
\begin{lemma}\label{lemma1}
For a strongly unimodal continuous univariate symmetric (about $0$) distribution an $n$-principal set is symmetric (about $0$).
\end{lemma}
\begin{proof}
Let $P$ be a strongly unimodal continuous univariate symmetric distribution about $0$. Let $\ga:=\set{a_1<a_2<\cdots<a_n}$ be an $n$-principal set for $P$ with the $n$th mean squared distance $V_n(P)$, i.e., $V_n(P)=V(P; \ga)$. Since the principal points are the expected values of their own Voronoi regions, we have $a_0<a_1<a_2<\cdots<a_n<a_{n+1}$, where $a_0=-\infty$ and $a_{n+1}=\infty$. Again, $P$ is symmetric, and so for any $s, t\in \D R\uu \set{-\infty, \infty}$ and $a\in \D R$, we have
\[\int_{-s}^{-t}(x+a)^2 dP=\int_{s}^t(x-a)^2 dP.\]
Hence, by the definition of mean squared distance,
\begin{align*} \label{eq41}
V(P; \ga)& =\sum_{i=1}^{n}\int_{\frac{a_{i-1}+a_{i}}{2}}^{\frac{a_i+a_{i+1}}{2}}(x-a_i)^2dP
=\sum_{i=1}^{n}\int_{-\frac{a_{i-1}+a_{i}}{2}}^{-\frac{a_i+a_{i+1}}{2}}(x+a_i)^2dP\\
&\geq V(P; -\ga)=\sum_{i=1}^{n}\int_{-\frac{a_{i-1}+a_{i}}{2}}^{-\frac{a_i+a_{i+1}}{2}}(x+a_i)^2dP=\sum_{i=1}^{n}\int_{\frac{a_{i-1}+a_{i}}{2}}^{\frac{a_i+a_{i+1}}{2}}(x-a_i)^2dP\geq V(P; \ga),
\end{align*}
and thus, $V_n(P)=V(P; \ga)=V(P; -\ga)$, which yields the fact that $-\ga$ is an $n$-principal set for $P$ whenever $\ga$ is an $n$-principal set for $P$. Since $P$ is strongly unimodal, by Theorem~\ref{Th0}, the $n$-principal set is unique, and so $\ga=-\ga$, i.e., an $n$-principal set is symmetric.
\end{proof}

We now give the following lemma (also see \cite[Remark~5.3]{GL1}).
\begin{lemma} \label{lemma2}  Let $P$ be a continuous univariate symmetric distribution. Let $Q:=P(\cdot|[0, \infty))$, the one tailed version of $P$. Let $n=2k$ for some positive integer $k$. Then, $V_{n}(P)=V_k(Q)$, in other words,  $\ga\uu (-\ga)\in \C C_{n}(P)$ if and only if $\ga \in \C C_k(Q)$.
\end{lemma}

\begin{proof} Let $n=2k$ for some $k\in \D N$. Let $\ga:=\set{a_1<a_2< \cdots< a_k}\in \C C_k(Q)$. Since the principal points are the expected values of their own Voronoi regions, we have $0<a_1<a_2\cdots<a_k<\infty$. Again, $P$ is symmetric, and so for any $s, t\in \D R\uu \set{-\infty, \infty}$ and $a\in \D R$, we have
\[\int_{-s}^{-t}(x+a)^2 dP=\int_{s}^t(x-a)^2 dP.\]
Hence, by the definition of $n$-th mean squared distance,
\begin{align*}
V_n(P)\leq \int_{-\infty}^\infty \min_{a\in\ga\uu (-\ga)}(x-a)^2 dP=2\int_0^\infty\min_{a\in \ga} (x-a)^2 dP=\frac{1}{P([0, \infty))}\int_0^\infty\min_{a\in \ga} (x-a)^2 dP
\end{align*}
implying $V_n(P)\leq V_k(Q)$. Conversely, let $\gb\in \C C_n(P)$ be symmetric. Write $\ga:=\gb\ii [0, \infty)$. Then, $\te{card}(\ga)=k$. Thus,
\[V_{n}(P) =\int\min_{b\in \gb}(x-b)^2 dP=2 \int\min_{b\in \gd}(x-b)^2 dP=\int\min_{b\in \ga}(x-a)^2 dQ\geq V_k(Q).\]
Thus, we see that  $V_{n}(P)=V_k(Q)$, in other words,  $\ga\uu (-\ga)\in \C C_{n}(P)$ if and only if $\ga \in \C C_n(Q)$, which is the lemma.
\end{proof}

By Lemma~\ref{lemma1} and Lemma~\ref{lemma2}, the following proposition follows.
\begin{prop} Let $P$ be a strongly unimodal continuous univariate symmetric distribution.
Let $\set{a_1<a_2<\cdots<a_{n}}$ be an $n$-principal set for $P$. Then, if $n=2k$, we have $a_1=-a_{2k}, \, a_2=-a_{2k-1}, a_3=-a_{2k-2}, \, \cdots,  \, a_k=-a_{k+1}$. If $n=2k+1$, we have $a_1=-a_{2k+1}, \, a_2=-a_{2k}, a_3=-a_{2k-1}, \, \cdots,  \, a_k=-a_{k+2}$, and $a_{k+1}=0$.
\end{prop}

\section{Numerical Methods}
In all the problems we consider here, the domain, $\mathcal D \subset \R$, of the probability density function is either $[0, \, 1]$, $[0, \, \infty)$, or $(-\infty, \, \infty)$, though the methodology described in this section does not require such standard intervals.
For the sake of clarity, we denote the endpoints of the regions $M(a_j | \alpha_n)$ as
$$
m_j := \left\{\begin{array}{lr}
c & \te{ if } j = 0, \\
\ds {a_j + a_{j+1} \over 2} & \te{ if } 1 \le j \le n-1, \\
d & \te{ if } j = n,
\end{array}
\right.
$$
which depend continuously on the array $\alpha_n$.
We seek to solve, numerically, the set of $n$, nonlinear equations.
\begin{eqnarray}
a_j = E(X: X \in M(a_j | \alpha_n)) &:=&
{e(M(a_j | \alpha_n)) \over P(M(a_j | \alpha_n))}
\qquad \qquad
\mbox{ for } j = 1, \, 2, \, \ldots, \, n,
\label{eqn.nls}
\end{eqnarray}
where the unconditional expected value function $e(M(a_j | \alpha_n))$, and  probability function $P(M(a_j | \alpha_n))$ are defined by
\begin{eqnarray*}
e(M(a_j | \alpha_n))  &:=& \int_{m_{j-1}}^{m_j} x f(x) \, dx, \te{ and }   \\
P(M(a_j | \alpha_n)) &:=& \int_{m_{j-1}}^{m_j} f(x) \, dx.
\end{eqnarray*}
Solving the nonlinear system in (\ref{eqn.nls}) is equivalent to finding the root of
the function $g: \R^n \to \R^n$ whose $j^{th}$ entry is defined as the difference:
\begin{eqnarray}
g_j(\alpha) &:=& 
a_j \int_{m_{j-1}}^{m_j} f(x) \, dx - \int_{m_{j-1}}^{m_j} x f(x) \, dx
\qquad \qquad \mbox{ for } j = 1, \, 2, \ldots, n.
\label{eqn.g}
\end{eqnarray}
The entries of the solution vector $\alpha_n \in \R^n$ are $n$ principal points.
Thus, we can apply Newton's algorithm for computing roots of nonlinear systems (for example, see \cite{K1} for a thorough guide to Newton's method) to obtain high precision numerical solutions to the optimal sets.
Given an initial vector $\alpha_0 \in \R^n$, the Newton iteration for finding the root to $g(\alpha)$ takes the form
\begin{eqnarray}
\alpha_{new} &=& \alpha_{old} + J(\alpha_{old})^{-1}  g(\alpha_{old}),
\label{eqn.newton}
\end{eqnarray}
where $J: \R^n \to \R^{n\times n}$ is the Jacobian matrix, whose entries are defined as
$J_{j,k} := \partial g_j / \partial a_k.$
The iteration is continued until the residual $\| g(\alpha_{new}) \|$ is sufficiently small. Note that the function $g_j(\alpha)$, for $j = 1, \, 2, \ldots, n$ depends only on $a_{j-1}, \, a_j, \, a_{j+1}$, indicating that the matrix $J(\alpha)$ is always tridiagonal.
Let $\ell_j$ describe the distance between consecutive points:
$$
\ell_j := a_{j+1} - a_j
\qquad \te{ for } j = 1, \, 2, \, \ldots, n-1.
$$
Then, the diagonal entries are given by
\begin{eqnarray*}
J_{j,j}(\alpha) &:=& \diff{g_j}{a_j} =
\int_{m_{j-1}}^{m_j} f(x) \, dx -
f\left(m_{j-1}\right) {\ell_{j-1} \over 4} -
f\left(m_j\right) {\ell_j \over 4}
\end{eqnarray*}
For the sake of simplicity, define $\ell_0 = \ell_n = 0$. In addition to being tridiagonal, the Jacobian is also symmetric: $J_{j+1,j} := \partial g_{j+1} / \partial a_j = \partial g_{j} / \partial a_{j+1} =: J_{j,j+1}$. The off-diagonal entries are given by
\begin{eqnarray*}
J_{j+1,j} = J_{j,j+1} =
- f\left(m_j \right) {\ell_j \over 4}
\qquad \qquad
\mbox{ for } j = 1, \, 2, \ldots, \, n-1
\end{eqnarray*}
Note that when $n = 1$, there is in fact no system to solve, as $\alpha_1$ is given explicitly by the expected value over the entire domain:
$$
\alpha_1 = \int_{\mathcal D} x f(x) \, dx. \\
$$

\subsection{Symmetric Probability Densities} \label{subsec.sym}
Consider the case that the domain of the probability distribution function $f(x)$ is symmetric about $0$, that is $c = -d$ and $f(x)$ is even.
In such a case, the number of unknowns can be reduced by taking into account this symmetry. Suppose that $\alpha_n \in \R^n$ is a root of the function g defined in (\ref{eqn.g}), then $a_j = -a_{n+1-j}$. Thus, we can reduce the size of the problem by half. That is, solving (\ref{eqn.g}) for a root $\alpha_n$ is equivalent to solving the same problem for $n/2$ or $(n-1)/2$ points, depending on whether $n$ is even or odd, over the half domain $[0, d)$. We outline here how the reformulation changes depending on whether $n$ is even or odd and the special cases of $n = 2$ and $n=3$.

${\bf n = 2}$: In the case of $n = 2$, the problem is solved by $\alpha_2 = \{ -\varphi, \, \varphi \}$, where $\varphi$ is twice the expected value of half the domain:
$$
\varphi = 2 \int_0^d x f(x) \, dx. \\
$$

${\bf n = 3}$: In the case of $n = 3$, the problem is solved by $\alpha_3 = \{ -\varphi, \, 0, \, \varphi \}$, where $\varphi$ is the root of the scalar valued function
$$
g(a) = a \int_{a/2}^d f(x) \, dx - \int_{a/2}^d x f(x) \, dx.
$$
Here, the Newton iteration for finding a root of the scalar valued function $g(a)$ takes the form
\begin{eqnarray*}
a_{new} &=& a_{old} - { g(a_{old}) \over g'(a_{old}) },
\end{eqnarray*}
where
\begin{eqnarray*}
g'(a) &=&
\int_{a/2}^b f(x) \,dx - {a \over 4} f\left( a  / 2 \right).
\end{eqnarray*}

${\bf n > 2}$ {\bf and even}: For $n$ even, let  $\widetilde a_j := a_{{n \over 2} + j}$ for $j = 1, \, 2, \ldots, n/2$. Then, the problem of computing principal points over the domain $\mathcal D$ is equivalent to solving (\ref{eqn.nls}) over the half domain $[0, \, d]$ for $\widetilde \alpha_{n \over 2} = \{\widetilde a_1, \, \widetilde a_2, \, \ldots, \, \widetilde a_{n \over 2} \}$. The principal points are constructed by
$$
\alpha_n = \{ - \widetilde a_{n \over 2}, \, -\widetilde a_{{n \over 2} - 1}, \, \ldots, \, -\widetilde a_1, \, \widetilde a_1, \ldots, \, \widetilde a_{{n \over 2} - 1}, \, \widetilde a_{n \over 2} \}.
$$

An example of this can be seen in the positive entries of the solution $\alpha_n$, when $n$ is even, for the double sided exponential distribution with domain $(-\infty, \, \infty)$ (see Tables \ref{tab.exp1} and \ref{tab.exp2}). These positive entries are exactly the solution $\alpha_{n \over 2}$ of the single sided exponential distribution with domain $[0, \infty)$.

${\bf n > 3}$ {\bf and odd}: For $n$ odd, let $\widetilde a_j := a_{{n+1 \over 2} + j}$ for $j = 1, \, 2, \ldots, (n-1)/2$ (in this case, $a_{n+1 \over 2} = 0$). Then, one can compute the principal points over the domain $\mathcal D$ by setting the nonlinear system in (\ref{eqn.nls}) over the variable domain $[m_0, \, d]$ for $\widetilde \alpha_{n-1 \over 2}$, where the left endpoint of the domain varies: $m_0 = \widetilde a_1/2$. This implies that the top left entry of the Jacobian matrix must be adjusted to
$$
J_{1,1}(\widetilde \alpha) 
= \int_{m_0}^{m_1} f(x) \, dx - f(m_0) {\widetilde a_1 \over 4} - f(m_1) {\ell_1 \over 4}.
$$
The principal points are then constructed by
$$
\alpha_n = \{ - \widetilde a_{n-1 \over 2}, \, -\widetilde a_{{n-1 \over 2} - 1}, \, \ldots, \, -\widetilde a_1, \, 0, \,  \widetilde a_1, \ldots, \, \widetilde a_{{n-1 \over 2} - 1}, \, \widetilde a_{n-1 \over 2} \}.
$$



\section{Numerical Results}
All of the results in this section were computed by applying the Newton iteration (\ref{eqn.newton}) until the maximal entry of the residual reached a tolerance of $10^{-15}$. That is, for the computed principal points $\alpha_n$, the following condition is met $| g_j(\alpha_n) | < 10^{-15}$ for each entry $j = 1, \, 2, \, \ldots, \, n$. In the tables in this section, only the first five digits of each principal points is shown, for $n = 1, \, 2, \, \ldots, \, 16$. It is possible to compute these values for significantly higher values of $n$ to the same high precision.
If the integral in the probability or expected value function cannot be explicitly computed for a given distribution function, then Matlab's built in command {\tt integral} is used, which approximates definite integrals using adaptive quadrature routines \cite{ML}.


For each of the Newton iteration runs, we used the initial distributions of $a_j = 1 + (j-1)/(n-1)$ if $\mathcal D = [0, \, \infty)$ and $a_j = j(d-c)/(n+1)$ if $\mathcal D = [c, \, d]$, 
for $j = 1, \, 2, \, \ldots, \, n$. For the case of $n=3$, the problem is reformulated as described in section~\ref{subsec.sym} and the initial point $a = 1$ if $d = \infty$ and $a = d/2$, otherwise.

\subsection{Computational Performance}
We show in the left side of Figure~\ref{fig.iter} that the Newton iteration requires relatively few iterations to converge and this method for approximating principal points to high precision is very efficient. While the number of iterations required for convergence grows with $n$, we see in the right side of Figure~\ref{fig.iter} that this growth is very mild.



%
%

\begin{figure}[tbhp]
\centering
\includegraphics[width=2.85in]{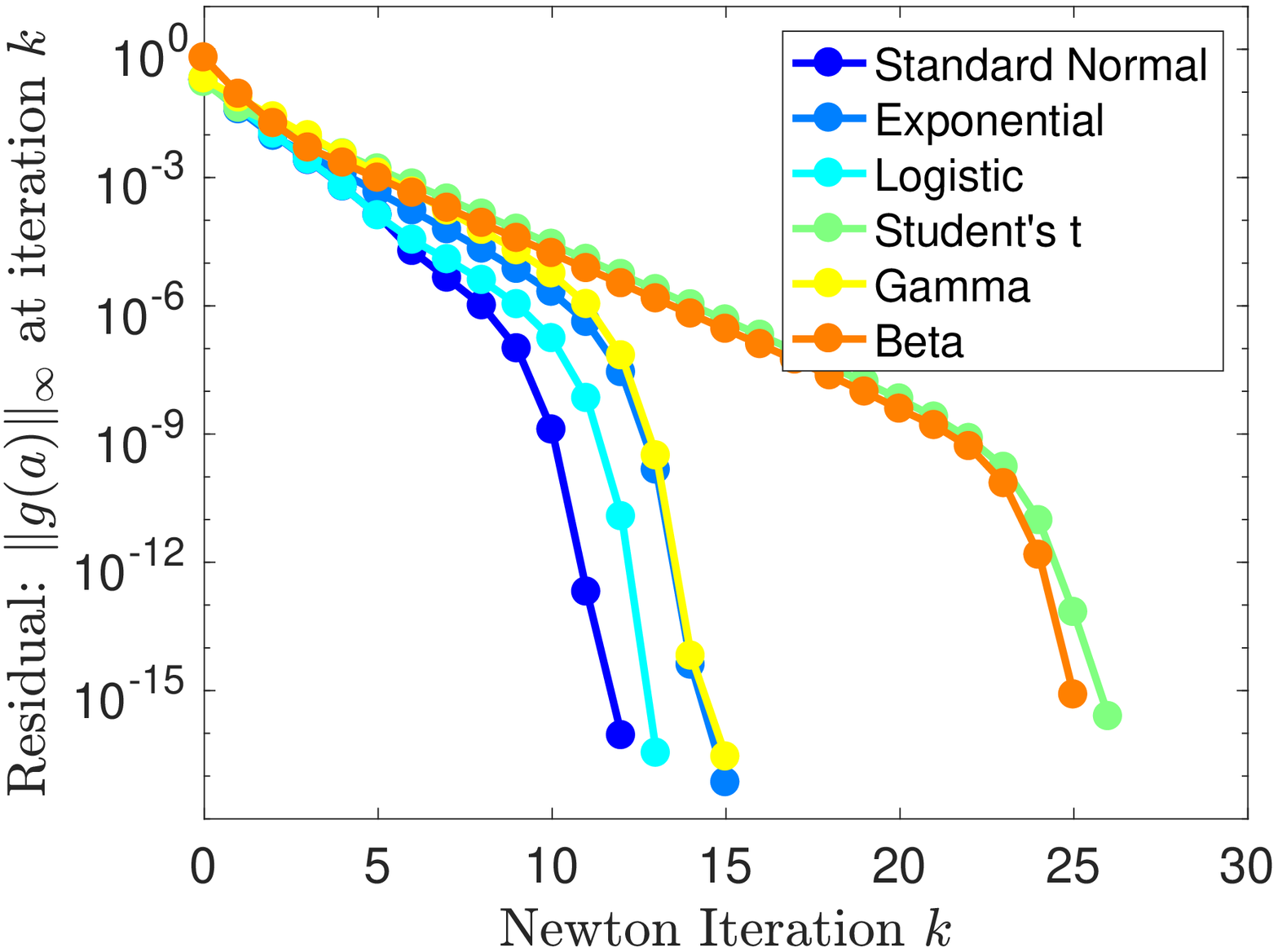} \hspace{1em}
\includegraphics[width=2.85in]{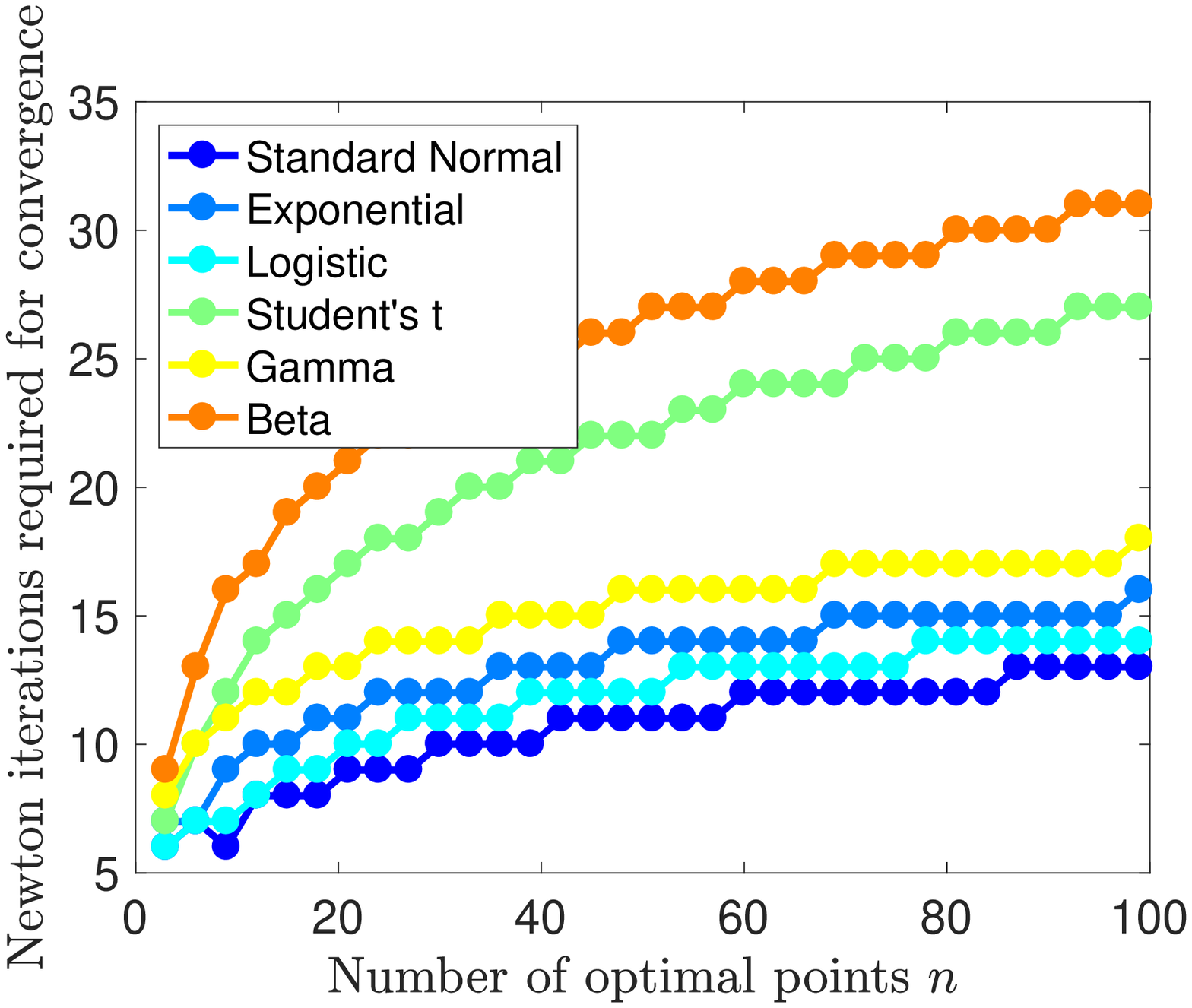}
\caption{{\bf Left:} Residuals for several probability distributions to compute 100 principal points. {\bf Right} Iterations required for a tolerance of $10^{-15}$ for several probability distributions as a function of the number of principal points $n$.}
\label{fig.iter}
\end{figure}

The computational efficiency of this method allows us to explore ideas in optimal quantizers. One such avenue of exploration is the idea of convergence of probability distribution. For example, it is well known that student's t-distribution, defined as
\begin{eqnarray*}
f_k(x) =
{
\Gamma \left( {k + 1 \over 2} \right) \over
\sqrt{k \pi} \, \Gamma \left( { k \over 2} \right)
}
\left( 1 + {x^2 \over k} \right)^{-{k+1 \over 2}}
\end{eqnarray*}
converges weakly to the normal distribution:
\begin{eqnarray*}
f(x) &=&
{1 \over
\sqrt{2 \pi} }
e^{-{x^2 \over 2}}.
\end{eqnarray*}
Does it follow then that the principal points, $\alpha_n(k)$, of student's t-distribution converge to those of the normal distribution? We investigate numerically whether that appears to be the case. We demonstrate in Figure~\ref{fig.TtoNormal} that this assertion appears to be true. In fact, this is supported by \cite{P1, P2}:
\begin{theorem} (see \cite{P2}, 6 Corrolary)
Suppose that the probability distribution $f_k(x)$ converges weakly to $f(x)$ in the limit as $k \to \infty$ and that $\alpha_n(k)$ are the principal points of the distribution $f_k(x)$ and $\beta_n$ are the principal points of $f(x)$. Then, after a suitable labeling, $\alpha_n(k)$ converges to $\beta_n$ in the limit as $k \to \infty$ for all values of $n$.
\end{theorem}

\begin{center}
\begin{figure}[h]
\includegraphics[width=3in]{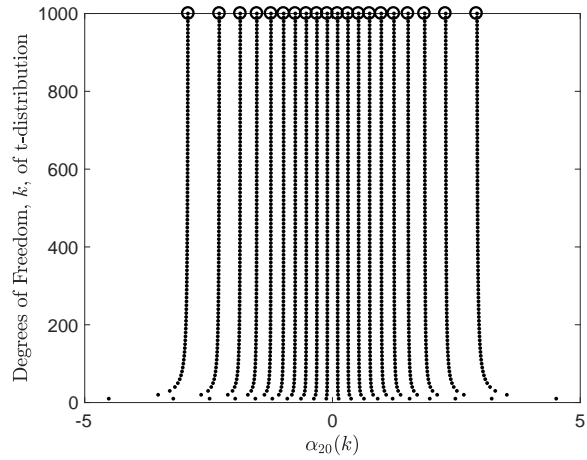}
\caption{Here we illustrate that the principal points of the t-distribution, depicted as black dots according to the degrees of freedom $k$, converge to those of the normal distribution depicted as circles at the top of the plot, in the limit as $k \to \infty$.}
\label{fig.TtoNormal}
\end{figure}
\end{center}


\appendix
\section{Tables}
In this appendix, we present principal points $\alpha_n$ for $n = 1, \, 2, \, \ldots, \, 16$ for several probability distributions. Although we only show four decimal places of precision, our numerical experiments were run to a precision of $10^{-15}$.

\subsection{Normal Distribution}
Table~\ref{tab.norm} presents principal points of the normal distribution
\begin{eqnarray*}
f(x) &=&
{1 \over
\sqrt{2 \pi} }
e^{-{x^2 \over 2}},
\end{eqnarray*}
defined over the domain $(-\infty, \, \infty)$. Note in the table that, since the distribution and domain are symmetric, the principal points are then symmetric about the mean.

\subsection{One-Sided Exponential Distribution}
Table~\ref{tab.exp1} presents principal points of the one-sided exponential distribution
\begin{eqnarray*}
f(x) &=& e^{-x},
\end{eqnarray*}
defined over the domain $[0, \, \infty)$.

\subsection{Double Exponential Distribution}
Table~\ref{tab.exp2} presents principal points of the double exponential distribution
\begin{eqnarray*}
f(x) &=& {1 \over 2} e^{-|x|},
\end{eqnarray*}
defined over the domain $(-\infty, \, \infty)$. 
Note that one can extract from the positive principal points of the even values of $n$ the principal points of the single sided exponential distribution $f(x) = e^{-x}$ over the domain $[0, \, \infty)$ for $n/2$. This is due to the symmetry property discussed in section~\ref{subsec.sym}. While the single-sided and double-sided distribution functions differ only in the coefficient $1/2$ over the intersection of their domains, it is clear from (\ref{eqn.nls}) that the coefficient does not affect the value of the principal points. 
One can observe this by comparing the values in Table~\ref{tab.exp1} and Table~\ref{tab.exp2}.


\subsection{Beta Distribution of the First Kind}
Table~\ref{tab.beta} presents principal points of the beta distribution of the first kind
\begin{eqnarray*}
f(x) &=& {1 \over \beta(r,\, s) } x^{r-1} (1-x)^{s-1}
\end{eqnarray*}
over the domain $[0, \, 1]$, where
\begin{eqnarray*}
\beta(r,\, s) &:=& \int_0^1 x^{r-1} (1-x)^{s-1} \, dx.
\end{eqnarray*}
In the table, the parameters are set to $r = 2$ and $s = 2$.

\subsection{Beta Distribution of the Second Kind}
Table~\ref{tab.beta2} presents principal points of the beta distribution of the second kind
\begin{eqnarray*}
f(x) &=& {1 \over \beta(r,\, s) } {x^{r-1}  \over (1+x)^{r+s}}
\end{eqnarray*}
over the domain $[0, \, \infty)$.
We require that $r > 0$ and $s > 2$ for $f(x)$ to be a probability density function.
Below are the results for $r = 1$ and $s = 3$.\\[1em]

\subsection{Gamma Distribution}
Table~\ref{tab.gamma} presents principal points of the gamma distribution.
\begin{eqnarray*}
f(x) &=& {1 \over a^b \Gamma(b) } x^{b-1} e^{-x/a}
\end{eqnarray*}
over the domain $[0, \, \infty)$.
Below are the results for $a = 1/\sqrt{2}$ and $b = 2$.

\subsection{Logistic Distribution}
Table~\ref{tab.log} presents principal points of the logistic distribution
\begin{eqnarray*}
f(x) &=& { e^{-|x|/a} \over a(1 + e^{-|x|/a})^2}
\end{eqnarray*}
over the domain $(-\infty, \, \infty)$.
The results in the table are for the parameter value set to $a = \sqrt{3}/\pi$.

\subsection{Student's t-distribution}
Table~\ref{tab.t} presents principal points of student's t-distribution function
\begin{eqnarray*}
f(x) &=& { \Gamma\left({k+1 \over 2}\right) \over \sqrt{k \pi} \Gamma\left({k \over 2}\right)}
\left( 1 + {x^2 \over k} \right)^{-\left( {k+1 \over 2} \right)}
\end{eqnarray*}
over the interval $(-\infty, \, \infty)$.
The table demonstrates results for $k=3$. Note that we require $k \ge 3$ for the probability, expected value, and variance formula to be well defined.

\pagebreak

\subsection{Tables of Computed Principal points} \hspace{1em}


\begin{table}[tbhp]
\centering
\begin{tabular}{|r|r|r|r|r|r|r|r|r|} 
 \hline 
$n$  & 1  & 2  & 3  & 4  & 5  & 6  & 7  & 8 \\ 
 \hline 
 $a_{1}$& 0 &  -0.7979 &  -1.2240 &  -1.5104 &  -1.7241 &  -1.8936 &  -2.0334 &  -2.1519 \\ 
 $a_{2}$& &   0.7979 & 0 &  -0.4528 &  -0.7646 &  -1.0001 &  -1.1881 &  -1.3439 \\ 
 $a_{3}$& & &   1.2240 &   0.4528 & 0 &  -0.3177 &  -0.5606 &  -0.7560 \\ 
 $a_{4}$& & & &   1.5104 &   0.7646 &   0.3177 & 0 &  -0.2451 \\ 
 $a_{5}$& & & & &   1.7241 &   1.0001 &   0.5606 &   0.2451 \\ 
 $a_{6}$& & & & & &   1.8936 &   1.1881 &   0.7560 \\ 
 $a_{7}$& & & & & & &   2.0334 &   1.3439 \\ 
 $a_{8}$& & & & & & & &   2.1519 \\ 
\hline 
 $V_n(P)$ &   1.0000 &   0.3634 &   0.1902 &   0.1175 &   0.0799 &   0.0580 &   0.0440 &   0.0345 \\ \hline \hline 
n  & 9  & 10  & 11  & 12  & 13  & 14  & 15  & 16 \\ 
 \hline  
 $a_{1}$&  -2.2547 &  -2.3451 &  -2.4257 &  -2.4984 &  -2.5645 &  -2.6251 &  -2.6809 &  -2.7326 \\ 
$ a_{2}$&  -1.4764 &  -1.5913 &  -1.6926 &  -1.7830 &  -1.8645 &  -1.9386 &  -2.0065 &  -2.0690 \\ 
 $a_{3}$&  -0.9188 &  -1.0578 &  -1.1788 &  -1.2857 &  -1.3813 &  -1.4675 &  -1.5461 &  -1.6180 \\ 
$ a_{4}$&  -0.4436 &  -0.6099 &  -0.7524 &  -0.8768 &  -0.9869 &  -1.0856 &  -1.1749 &  -1.2562 \\ 
 $a_{5}$& 0 &  -0.1996 &  -0.3675 &  -0.5118 &  -0.6383 &  -0.7504 &  -0.8511 &  -0.9423 \\ 
 $a_{6}$&   0.4436 &   0.1996 & 0 &  -0.1684 &  -0.3138 &  -0.4413 &  -0.5548 &  -0.6568 \\ 
 $a_{7}$&   0.9188 &   0.6099 &   0.3675 &   0.1684 & 0 &  -0.1457 &  -0.2739 &  -0.3880 \\ 
 $a_{8}$&   1.4764 &   1.0578 &   0.7524 &   0.5118 &   0.3138 &   0.1457 & 0 &  -0.1284 \\ 
 $a_{9}$&   2.2547 &   1.5913 &   1.1788 &   0.8768 &   0.6383 &   0.4413 &   0.2739 &   0.1284 \\ 
 $a_{10}$& &   2.3451 &   1.6926 &   1.2857 &   0.9869 &   0.7504 &   0.5548 &   0.3880 \\ 
 $a_{11}$& & &   2.4257 &   1.7830 &   1.3813 &   1.0856 &   0.8511 &   0.6568 \\ 
 $a_{12}$& & & &   2.4984 &   1.8645 &   1.4675 &   1.1749 &   0.9423 \\ 
 $a_{13}$& & & & &   2.5645 &   1.9386 &   1.5461 &   1.2562 \\ 
 $a_{14}$& & & & & &   2.6251 &   2.0065 &   1.6180 \\ 
 $a_{15}$& & & & & & &   2.6809 &   2.0690 \\ 
 $a_{16}$& & & & & & & &   2.7326 \\ 
\hline 
 $V_n(P)$ &   0.0279 &   0.0229 &   0.0192 &   0.0163 &   0.0141 &   0.0122 &   0.0107 &   0.0095 \\ 
\hline 
 \end{tabular} 

\caption{Principal points for $n = 1$ through $n = 16$ for the {\bf normal distribution} $f(x) = (1/\sqrt{2 \pi}) e^{-x^2/2}$ on the domain $(-\infty, \, \infty)$.}
\label{tab.norm}
\end{table}
\pagebreak

\hspace{1em}


\begin{table}[tbhp]
\centering
$$
\begin{array}{|r|r|r|r|r|r|r|r|r|} 
 \hline 
n  & 1  & 2  & 3  & 4  & 5  & 6  & 7  & 8 \\ 
 \hline 
 a_{1}&   1.0000 &   0.5936 &   0.4240 &   0.3301 &   0.2704 &   0.2290 &   0.1986 &   0.1753 \\ 
 a_{2}& &   2.5936 &   1.6112 &   1.1780 &   0.9305 &   0.7697 &   0.6565 &   0.5725 \\ 
 a_{3}& & &   3.6112 &   2.3652 &   1.7784 &   1.4298 &   1.1972 &   1.0305 \\ 
 a_{4}& & & &   4.3652 &   2.9657 &   2.2777 &   1.8574 &   1.5712 \\ 
 a_{5}& & & & &   4.9657 &   3.4650 &   2.7053 &   2.2313 \\ 
 a_{6}& & & & & &   5.4650 &   3.8925 &   3.0792 \\ 
 a_{7}& & & & & & &   5.8925 &   4.2665 \\ 
 a_{8}& & & & & & & &   6.2665 \\ 
\hline 
 V_n(P) &   1.0000 &   0.3524 &   0.1797 &   0.1090 &   0.0731 &   0.0524 &   0.0394 &   0.0307 \\ \hline \hline 
n  & 9  & 10  & 11  & 12  & 13  & 14  & 15  & 16 \\ 
 \hline 
 a_{1}&   0.1570 &   0.1421 &   0.1298 &   0.1194 &   0.1106 &   0.1030 &   0.0964 &   0.0906 \\ 
 a_{2}&   0.5077 &   0.4560 &   0.4140 &   0.3790 &   0.3495 &   0.3243 &   0.3025 &   0.2834 \\ 
 a_{3}&   0.9048 &   0.8067 &   0.7279 &   0.6632 &   0.6091 &   0.5632 &   0.5237 &   0.4894 \\ 
 a_{4}&   1.3628 &   1.2039 &   1.0786 &   0.9771 &   0.8933 &   0.8227 &   0.7626 &   0.7107 \\ 
 a_{5}&   1.9035 &   1.6618 &   1.4758 &   1.3278 &   1.2072 &   1.1069 &   1.0222 &   0.9496 \\ 
 a_{6}&   2.5636 &   2.2025 &   1.9337 &   1.7250 &   1.5579 &   1.4208 &   1.3063 &   1.2091 \\ 
 a_{7}&   3.4115 &   2.8627 &   2.4744 &   2.1829 &   1.9551 &   1.7715 &   1.6203 &   1.4933 \\ 
 a_{8}&   4.5988 &   3.7106 &   3.1346 &   2.7236 &   2.4130 &   2.1687 &   1.9710 &   1.8073 \\ 
 a_{9}&   6.5988 &   4.8979 &   3.9825 &   3.3838 &   2.9537 &   2.6267 &   2.3681 &   2.1579 \\ 
 a_{10}& &   6.8979 &   5.1697 &   4.2317 &   3.6139 &   3.1674 &   2.8261 &   2.5551 \\ 
 a_{11}& & &   7.1697 &   5.4189 &   4.4618 &   3.8275 &   3.3668 &   3.0131 \\ 
 a_{12}& & & &   7.4189 &   5.6490 &   4.6754 &   4.0269 &   3.5538 \\ 
 a_{13}& & & & &   7.6490 &   5.8627 &   4.8749 &   4.2139 \\ 
 a_{14}& & & & & &   7.8627 &   6.0621 &   5.0618 \\ 
 a_{15}& & & & & & &   8.0621 &   6.2491 \\ 
 a_{16}& & & & & & & &   8.2491 \\ 
\hline 
 V_n(P) &   0.0246 &   0.0202 &   0.0168 &   0.0143 &   0.0122 &   0.0106 &   0.0093 &   0.0082 \\ 
\hline 
 \end{array} 

$$
\caption{Principal points for $n = 1$ through $n = 16$ for the {\bf exponential distribution} $
f(x) = e^{-x}$ on the domain $[0, \, \infty)$.}
\label{tab.exp1}
\end{table}

\pagebreak

\hspace{1em}


\begin{table}[tbhp]
\centering
$$
\begin{array}{|r|r|r|r|r|r|r|r|r|} 
 \hline 
n  & 1  & 2  & 3  & 4  & 5  & 6  & 7  & 8 \\ 
 \hline 
 a_{1}& 0 &  -1.0000 & -2 &  -2.5936 &  -3.1872 &  -3.6112 &  -4.0352 &  -4.3652 \\ 
 a_{2}& &   1.0000 & 0 &  -0.5936 &  -1.1872 &  -1.6112 &  -2.0352 &  -2.3652 \\ 
 a_{3}& & & 2 &   0.5936 & 0 &  -0.4240 &  -0.8479 &  -1.1780 \\ 
 a_{4}& & & &   2.5936 &   1.1872 &   0.4240 & 0 &  -0.3301 \\ 
 a_{5}& & & & &   3.1872 &   1.6112 &   0.8479 &   0.3301 \\ 
 a_{6}& & & & & &   3.6112 &   2.0352 &   1.1780 \\ 
 a_{7}& & & & & & &   4.0352 &   2.3652 \\ 
 a_{8}& & & & & & & &   4.3652 \\ 
\hline 
 V_n(P) &   2.0000 &   1.0000 &   0.5285 &   0.3524 &   0.2396 &   0.1797 &   0.1362 &   0.1090 \\ \hline \hline 
n  & 9  & 10  & 11  & 12  & 13  & 14  & 15  & 16 \\ 
 \hline 
 a_{1}&  -4.6953 &  -4.9657 &  -5.2360 &  -5.4650 &  -5.6940 &  -5.8925 &  -6.0911 &  -6.2665 \\ 
 a_{2}&  -2.6953 &  -2.9657 &  -3.2360 &  -3.4650 &  -3.6940 &  -3.8925 &  -4.0911 &  -4.2665 \\ 
 a_{3}&  -1.5081 &  -1.7784 &  -2.0488 &  -2.2777 &  -2.5067 &  -2.7053 &  -2.9039 &  -3.0792 \\ 
 a_{4}&  -0.6602 &  -0.9305 &  -1.2009 &  -1.4298 &  -1.6588 &  -1.8574 &  -2.0560 &  -2.2313 \\ 
 a_{5}& 0 &  -0.2704 &  -0.5407 &  -0.7697 &  -0.9986 &  -1.1972 &  -1.3958 &  -1.5712 \\ 
 a_{6}&   0.6602 &   0.2704 & 0 &  -0.2290 &  -0.4579 &  -0.6565 &  -0.8551 &  -1.0305 \\ 
 a_{7}&   1.5081 &   0.9305 &   0.5407 &   0.2290 & 0 &  -0.1986 &  -0.3972 &  -0.5725 \\ 
 a_{8}&   2.6953 &   1.7784 &   1.2009 &   0.7697 &   0.4579 &   0.1986 & 0 &  -0.1753 \\ 
 a_{9}&   4.6953 &   2.9657 &   2.0488 &   1.4298 &   0.9986 &   0.6565 &   0.3972 &   0.1753 \\ 
 a_{10}& &   4.9657 &   3.2360 &   2.2777 &   1.6588 &   1.1972 &   0.8551 &   0.5725 \\ 
 a_{11}& & &   5.2360 &   3.4650 &   2.5067 &   1.8574 &   1.3958 &   1.0305 \\ 
 a_{12}& & & &   5.4650 &   3.6940 &   2.7053 &   2.0560 &   1.5712 \\ 
 a_{13}& & & & &   5.6940 &   3.8925 &   2.9039 &   2.2313 \\ 
 a_{14}& & & & & &   5.8925 &   4.0911 &   3.0792 \\ 
 a_{15}& & & & & & &   6.0911 &   4.2665 \\ 
 a_{16}& & & & & & & &   6.2665 \\ 
\hline 
 V_n(P) &   0.0877 &   0.0731 &   0.0612 &   0.0524 &   0.0451 &   0.0394 &   0.0346 &   0.0307 \\ 
\hline 
 \end{array} 

$$
\caption{Principal points for $n = 1$ through $n = 16$ for the {\bf double exponential distribution} $f(x) = (1/2) e^{-|x|}$ on the domain $(-\infty, \, \infty)$. Note that one can extract from the positive entries of the even values of $n$ the principal points of the single exponential distribution for $n/2$, shown in Table~\ref{tab.exp1}}
\label{tab.exp2}
\end{table}

\pagebreak


\hspace{1em}


\begin{table}[tbhp]
\centering
$$
\begin{array}{|r|r|r|r|r|r|r|r|r|} 
 \hline 
n  & 1  & 2  & 3  & 4  & 5  & 6  & 7  & 8 \\ 
 \hline 
 a_{1}&   0.5000 &   0.3125 &   0.2351 &   0.1914 &   0.1630 &   0.1428 &   0.1276 &   0.1157 \\ 
 a_{2}& &   0.6875 &   0.5000 &   0.4011 &   0.3386 &   0.2949 &   0.2625 &   0.2374 \\ 
 a_{3}& & &   0.7649 &   0.5989 &   0.5000 &   0.4328 &   0.3836 &   0.3458 \\ 
 a_{4}& & & &   0.8086 &   0.6614 &   0.5672 &   0.5000 &   0.4491 \\ 
 a_{5}& & & & &   0.8370 &   0.7051 &   0.6164 &   0.5509 \\ 
 a_{6}& & & & & &   0.8572 &   0.7375 &   0.6542 \\ 
 a_{7}& & & & & & &   0.8724 &   0.7626 \\ 
 a_{8}& & & & & & & &   0.8843 \\ 
\hline 
 V_n(P) &   0.0500 &   0.0148 &   0.0071 &   0.0041 &   0.0027 &   0.0019 &   0.0014 &   0.0011 \\ \hline \hline 
n  & 9  & 10  & 11  & 12  & 13  & 14  & 15  & 16 \\ 
 \hline 
 a_{1}&   0.1061 &   0.0982 &   0.0916 &   0.0859 &   0.0809 &   0.0766 &   0.0728 &   0.0694 \\ 
 a_{2}&   0.2172 &   0.2006 &   0.1867 &   0.1749 &   0.1646 &   0.1557 &   0.1478 &   0.1408 \\ 
 a_{3}&   0.3157 &   0.2911 &   0.2705 &   0.2531 &   0.2380 &   0.2249 &   0.2134 &   0.2031 \\ 
 a_{4}&   0.4089 &   0.3763 &   0.3492 &   0.3262 &   0.3065 &   0.2894 &   0.2743 &   0.2610 \\ 
 a_{5}&   0.5000 &   0.4590 &   0.4252 &   0.3967 &   0.3723 &   0.3511 &   0.3326 &   0.3162 \\ 
 a_{6}&   0.5911 &   0.5410 &   0.5000 &   0.4657 &   0.4365 &   0.4113 &   0.3892 &   0.3698 \\ 
 a_{7}&   0.6843 &   0.6237 &   0.5748 &   0.5343 &   0.5000 &   0.4705 &   0.4448 &   0.4222 \\ 
 a_{8}&   0.7828 &   0.7089 &   0.6508 &   0.6033 &   0.5635 &   0.5295 &   0.5000 &   0.4741 \\ 
 a_{9}&   0.8939 &   0.7994 &   0.7295 &   0.6738 &   0.6277 &   0.5887 &   0.5552 &   0.5259 \\ 
 a_{10}& &   0.9018 &   0.8133 &   0.7469 &   0.6935 &   0.6489 &   0.6108 &   0.5778 \\ 
 a_{11}& & &   0.9084 &   0.8251 &   0.7620 &   0.7106 &   0.6674 &   0.6302 \\ 
 a_{12}& & & &   0.9141 &   0.8354 &   0.7751 &   0.7257 &   0.6838 \\ 
 a_{13}& & & & &   0.9191 &   0.8443 &   0.7866 &   0.7390 \\ 
 a_{14}& & & & & &   0.9234 &   0.8522 &   0.7969 \\ 
 a_{15}& & & & & & &   0.9272 &   0.8592 \\ 
 a_{16}& & & & & & & &   0.9306 \\ 
\hline 
 V_n(P) &   0.0009 &   0.0007 &   0.0006 &   0.0005 &   0.0004 &   0.0004 &   0.0003 &   0.0003 \\ 
\hline 
 \end{array} 

$$
\caption{Principal points for $n = 1$ through $n = 16$ for the beta distribution of the first kind $f(x) = \beta(r,\, s)^{-1} x^{r-1} (1-x)^{s-1}$, where $\beta(r,\, s) = \int_0^1 x^{r-1} (1-x)^{s-1} \, dx$ on the domain $[0, \, 1]$. For this table, the parameters $r = 2$ and $s = 2$ were used.}
\label{tab.beta}
\end{table}

\pagebreak

\hspace{1em}


\begin{table}[tbhp]
\centering
$$
\begin{array}{|r|r|r|r|r|r|r|r|r|} 
 \hline 
n  & 1  & 2  & 3  & 4  & 5  & 6  & 7  & 8 \\ 
 \hline 
 a_{1}&   0.5000 &   0.3660 &   0.2912 &   0.2426 &   0.2081 &   0.1824 &   0.1624 &   0.1464 \\ 
 a_{2}& &   3.0981 &   1.7822 &   1.2637 &   0.9819 &   0.8038 &   0.6809 &   0.5908 \\ 
 a_{3}& & &   7.3466 &   3.8776 &   2.6106 &   1.9591 &   1.5644 &   1.3004 \\ 
 a_{4}& & & &  13.6327 &   6.7797 &   4.3909 &   3.2068 &   2.5095 \\ 
 a_{5}& & & & &  22.3390 &  10.6159 &   6.6639 &   4.7572 \\ 
 a_{6}& & & & & &  33.8476 &  15.5134 &   9.4884 \\ 
 a_{7}& & & & & & &  48.5401 &  21.5994 \\ 
 a_{8}& & & & & & & &  66.7983 \\ 
\hline 
 V_n(P) &   0.7500 &   0.4019 &   0.2544 &   0.1765 &   0.1299 &   0.0998 &   0.0791 &   0.0643 \\ \hline \hline 
n  & 9  & 10  & 11  & 12  & 13  & 14  & 15  & 16 \\ 
 \hline 
 a_{1}&   0.1332 &   0.1223 &   0.1130 &   0.1051 &   0.0982 &   0.0921 &   0.0868 &   0.0820 \\ 
 a_{2}&   0.5218 &   0.4673 &   0.4232 &   0.3867 &   0.3560 &   0.3298 &   0.3073 &   0.2876 \\ 
 a_{3}&   1.1118 &   0.9704 &   0.8607 &   0.7731 &   0.7016 &   0.6421 &   0.5918 &   0.5488 \\ 
 a_{4}&   2.0538 &   1.7343 &   1.4987 &   1.3182 &   1.1757 &   1.0606 &   0.9656 &   0.8860 \\ 
 a_{5}&   3.6589 &   2.9541 &   2.4674 &   2.1131 &   1.8446 &   1.6348 &   1.4665 &   1.3288 \\ 
 a_{6}&   6.6428 &   5.0324 &   4.0142 &   3.3200 &   2.8200 &   2.4448 &   2.1538 &   1.9223 \\ 
 a_{7}&  12.9236 &   8.8960 &   6.6497 &   5.2471 &   4.3009 &   3.6259 &   3.1234 &   2.7366 \\ 
 a_{8}&  29.0012 &  17.0284 &  11.5491 &   8.5306 &   6.6656 &   5.4192 &   4.5373 &   3.8854 \\ 
 a_{9}&  89.0035 &  37.8458 &  21.8619 &  14.6346 &  10.6948 &   8.2828 &   6.6839 &   5.5605 \\ 
 a_{10}& & 115.5374 &  48.2605 &  27.4830 &  18.1849 &  13.1620 &  10.1116 &   8.1038 \\ 
 a_{11}& & & 146.7815 &  60.3724 &  33.9509 &  22.2323 &  15.9521 &  12.1650 \\ 
 a_{12}& & & & 183.1172 &  74.3087 &  41.3245 &  26.8093 &  19.0847 \\ 
 a_{13}& & & & & 224.9261 &  90.1965 &  49.6628 &  31.9482 \\ 
 a_{14}& & & & & & 272.5896 & 108.1631 &  59.0248 \\ 
 a_{15}& & & & & & & 326.4892 & 128.3355 \\ 
 a_{16}& & & & & & & & 387.0065 \\ 
\hline 
 V_n(P) &   0.0533 &   0.0449 &   0.0383 &   0.0331 &   0.0289 &   0.0254 &   0.0226 &   0.0202 \\ 
\hline 
 \end{array} 

$$
\caption{Principal points for $n = 1$ through $n = 16$ for the beta distribution of the second kind, $f(x) =  \beta(r,\, s)^{-1} {x^{r-1}  / (1+x)^{r+s}}$ on the domain $[0, \, \infty)$. For this table, the parameters $r = 1$ and $s = 3$ were used.}
\label{tab.beta2}
\end{table}
\pagebreak
\hspace{1em}


\begin{table}[tbhp]
\centering
$$
\begin{array}{|r|r|r|r|r|r|r|r|r|} 
 \hline 
n  & 1  & 2  & 3  & 4  & 5  & 6  & 7  & 8 \\ 
 \hline 
 a_{1}&   1.4142 &   0.9271 &   0.7108 &   0.5847 &   0.5009 &   0.4407 &   0.3950 &   0.3590 \\ 
 a_{2}& &   2.7353 &   1.8420 &   1.4269 &   1.1798 &   1.0136 &   0.8932 &   0.8014 \\ 
 a_{3}& & &   3.5501 &   2.4815 &   1.9577 &   1.6363 &   1.4157 &   1.2537 \\ 
 a_{4}& & & &   4.1445 &   2.9772 &   2.3842 &   2.0119 &   1.7523 \\ 
 a_{5}& & & & &   4.6135 &   3.3829 &   2.7420 &   2.3325 \\ 
 a_{6}& & & & & &   5.0012 &   3.7266 &   3.0505 \\ 
 a_{7}& & & & & & &   5.3318 &   4.0248 \\ 
 a_{8}& & & & & & & &   5.6199 \\ 
\hline 
 V_n(P) &   1.0000 &   0.3565 &   0.1836 &   0.1120 &   0.0755 &   0.0544 &   0.0410 &   0.0321 \\ \hline \hline 
n  & 9  & 10  & 11  & 12  & 13  & 14  & 15  & 16 \\ 
 \hline 
 a_{1}&   0.3299 &   0.3057 &   0.2853 &   0.2678 &   0.2526 &   0.2393 &   0.2275 &   0.2169 \\ 
 a_{2}&   0.7288 &   0.6698 &   0.6208 &   0.5794 &   0.5439 &   0.5130 &   0.4859 &   0.4619 \\ 
 a_{3}&   1.1289 &   1.0295 &   0.9482 &   0.8804 &   0.8228 &   0.7731 &   0.7299 &   0.6919 \\ 
 a_{4}&   1.5592 &   1.4092 &   1.2889 &   1.1898 &   1.1067 &   1.0359 &   0.9747 &   0.9213 \\ 
 a_{5}&   2.0433 &   1.8261 &   1.6561 &   1.5187 &   1.4051 &   1.3094 &   1.2275 &   1.1565 \\ 
 a_{6}&   2.6127 &   2.3003 &   2.0638 &   1.8774 &   1.7259 &   1.6001 &   1.4937 &   1.4023 \\ 
 a_{7}&   3.3218 &   2.8618 &   2.5307 &   2.2783 &   2.0782 &   1.9150 &   1.7788 &   1.6631 \\ 
 a_{8}&   4.2883 &   3.5641 &   3.0860 &   2.7396 &   2.4740 &   2.2624 &   2.0890 &   1.9438 \\ 
 a_{9}&   5.8753 &   4.5243 &   3.7830 &   3.2901 &   2.9308 &   2.6540 &   2.4324 &   2.2502 \\ 
 a_{10}& &   6.1046 &   4.7380 &   3.9825 &   3.4773 &   3.1071 &   2.8206 &   2.5905 \\ 
 a_{11}& & &   6.3126 &   4.9332 &   4.1660 &   3.6502 &   3.2706 &   2.9758 \\ 
 a_{12}& & & &   6.5030 &   5.1129 &   4.3357 &   3.8109 &   3.4232 \\ 
 a_{13}& & & & &   6.6786 &   5.2794 &   4.4936 &   3.9609 \\ 
 a_{14}& & & & & &   6.8414 &   5.4345 &   4.6413 \\ 
 a_{15}& & & & & & &   6.9932 &   5.5796 \\ 
 a_{16}& & & & & & & &   7.1354 \\ 
\hline 
 V_n(P) &   0.0258 &   0.0211 &   0.0177 &   0.0150 &   0.0129 &   0.0112 &   0.0098 &   0.0086 \\ 
\hline 
 \end{array} 

$$
\caption{Principal points for $n = 1$ through $n = 16$ for the gamma distribution $f(x) = (a^b \Gamma(b) )^{-1} x^{b-1} e^{-x/a}$ on the domain $[0, \, \infty)$. For this table, the parameters $a = 1/\sqrt{2}$ and $b = 2$ were used.}
\label{tab.gamma}
\end{table}

\pagebreak
%
\hspace{1em}


\begin{table}[tbhp]
\centering
$$
\begin{array}{|r|r|r|r|r|r|r|r|r|} 
 \hline 
n  & 1  & 2  & 3  & 4  & 5  & 6  & 7  & 8 \\ 
 \hline 
 a_{1}& 0 &  -0.7643 &  -1.2621 &  -1.6382 &  -1.9422 &  -2.1978 &  -2.4185 &  -2.6129 \\ 
 a_{2}& &   0.7643 & 0 &  -0.4569 &  -0.7947 &  -1.0671 &  -1.2971 &  -1.4971 \\ 
 a_{3}& & &   1.2621 &   0.4569 & 0 &  -0.3270 &  -0.5862 &  -0.8033 \\ 
 a_{4}& & & &   1.6382 &   0.7947 &   0.3270 & 0 &  -0.2548 \\ 
 a_{5}& & & & &   1.9422 &   1.0671 &   0.5862 &   0.2548 \\ 
 a_{6}& & & & & &   2.1978 &   1.2971 &   0.8033 \\ 
 a_{7}& & & & & & &   2.4185 &   1.4971 \\ 
 a_{8}& & & & & & & &   2.6129 \\ 
\hline 
 V_n(P) &   1.0000 &   0.4158 &   0.2307 &   0.1472 &   0.1022 &   0.0752 &   0.0576 &   0.0456 \\ \hline \hline 
n  & 9  & 10  & 11  & 12  & 13  & 14  & 15  & 16 \\ 
 \hline 
 a_{1}&  -2.7866 &  -2.9436 &  -3.0869 &  -3.2187 &  -3.3407 &  -3.4543 &  -3.5606 &  -3.6604 \\ 
 a_{2}&  -1.6743 &  -1.8337 &  -1.9787 &  -2.1117 &  -2.2345 &  -2.3488 &  -2.4555 &  -2.5558 \\ 
 a_{3}&  -0.9913 &  -1.1579 &  -1.3077 &  -1.4442 &  -1.5697 &  -1.6859 &  -1.7941 &  -1.8954 \\ 
 a_{4}&  -0.4658 &  -0.6473 &  -0.8073 &  -0.9510 &  -1.0816 &  -1.2017 &  -1.3129 &  -1.4165 \\ 
 a_{5}& 0 &  -0.2088 &  -0.3870 &  -0.5432 &  -0.6829 &  -0.8096 &  -0.9259 &  -1.0334 \\ 
 a_{6}&   0.4658 &   0.2088 & 0 &  -0.1769 &  -0.3311 &  -0.4685 &  -0.5926 &  -0.7062 \\ 
 a_{7}&   0.9913 &   0.6473 &   0.3870 &   0.1769 & 0 &  -0.1534 &  -0.2895 &  -0.4121 \\ 
 a_{8}&   1.6743 &   1.1579 &   0.8073 &   0.5432 &   0.3311 &   0.1534 & 0 &  -0.1355 \\ 
 a_{9}&   2.7866 &   1.8337 &   1.3077 &   0.9510 &   0.6829 &   0.4685 &   0.2895 &   0.1355 \\ 
 a_{10}& &   2.9436 &   1.9787 &   1.4442 &   1.0816 &   0.8096 &   0.5926 &   0.4121 \\ 
 a_{11}& & &   3.0869 &   2.1117 &   1.5697 &   1.2017 &   0.9259 &   0.7062 \\ 
 a_{12}& & & &   3.2187 &   2.2345 &   1.6859 &   1.3129 &   1.0334 \\ 
 a_{13}& & & & &   3.3407 &   2.3488 &   1.7941 &   1.4165 \\ 
 a_{14}& & & & & &   3.4543 &   2.4555 &   1.8954 \\ 
 a_{15}& & & & & & &   3.5606 &   2.5558 \\ 
 a_{16}& & & & & & & &   3.6604 \\ 
\hline 
 V_n(P) &   0.0370 &   0.0306 &   0.0257 &   0.0220 &   0.0189 &   0.0165 &   0.0145 &   0.0129 \\ 
\hline 
 \end{array} 

$$
\caption{Principal points for $n = 1$ through $n = 16$ for the logistic distribution $f(x) =  e^{-|x|/a} /  a(1 + e^{-|x|/a})^2$ on the domain $(-\infty, \, \infty)$.}
\label{tab.log}
\end{table}

\pagebreak
%
%
\hspace{1em}


\begin{table}[tbhp]
\centering
$$
\begin{array}{|r|r|r|r|r|r|r|r|r|} 
 \hline 
n  & 1  & 2  & 3  & 4  & 5  & 6  & 7  & 8 \\ 
 \hline 
 a_{1}& 0 &  -1.1027 &  -2.2865 &  -3.7616 &  -5.6124 &  -7.8980 & -10.6700 & -13.9774 \\ 
 a_{2}& &   1.1027 & 0 &  -0.7815 &  -1.5520 &  -2.4054 &  -3.3883 &  -4.5305 \\ 
 a_{3}& & &   2.2865 &   0.7815 & 0 &  -0.6173 &  -1.2074 &  -1.8255 \\ 
 a_{4}& & & &   3.7616 &   1.5520 &   0.6173 & 0 &  -0.5127 \\ 
 a_{5}& & & & &   5.6124 &   2.4054 &   1.2074 &   0.5127 \\ 
 a_{6}& & & & & &   7.8980 &   3.3883 &   1.8255 \\ 
 a_{7}& & & & & & &  10.6700 &   4.5305 \\ 
 a_{8}& & & & & & & &  13.9774 \\ 
\hline 
 V_n(P) &   3.0000 &   1.7841 &   1.2439 &   0.9302 &   0.7264 &   0.5849 &   0.4820 &   0.4045 \\ \hline \hline 
n  & 9  & 10  & 11  & 12  & 13  & 14  & 15  & 16 \\ 
 \hline 
 a_{1}& -17.8678 & -22.3881 & -27.5850 & -33.5049 & -40.1942 & -47.6992 & -56.0662 & -65.3412 \\ 
 a_{2}&  -5.8553 &  -7.3823 &  -9.1298 & -11.1146 & -13.3533 & -15.8620 & -18.6566 & -21.7529 \\ 
 a_{3}&  -2.5014 &  -3.2540 &  -4.0976 &  -5.0435 &  -6.1016 &  -7.2810 &  -8.5900 & -10.0367 \\ 
 a_{4}&  -0.9970 &  -1.4892 &  -2.0097 &  -2.5720 &  -3.1859 &  -3.8592 &  -4.5983 &  -5.4091 \\ 
 a_{5}& 0 &  -0.4392 &  -0.8526 &  -1.2651 &  -1.6918 &  -2.1428 &  -2.6253 &  -3.1452 \\ 
 a_{6}&   0.9970 &   0.4392 & 0 &  -0.3846 &  -0.7463 &  -1.1034 &  -1.4671 &  -1.8454 \\ 
 a_{7}&   2.5014 &   1.4892 &   0.8526 &   0.3846 & 0 &  -0.3422 &  -0.6644 &  -0.9803 \\ 
 a_{8}&   5.8553 &   3.2540 &   2.0097 &   1.2651 &   0.7463 &   0.3422 & 0 &  -0.3083 \\ 
 a_{9}&  17.8678 &   7.3823 &   4.0976 &   2.5720 &   1.6918 &   1.1034 &   0.6644 &   0.3083 \\ 
 a_{10}& &  22.3881 &   9.1298 &   5.0435 &   3.1859 &   2.1428 &   1.4671 &   0.9803 \\ 
 a_{11}& & &  27.5850 &  11.1146 &   6.1016 &   3.8592 &   2.6253 &   1.8454 \\ 
 a_{12}& & & &  33.5049 &  13.3533 &   7.2810 &   4.5983 &   3.1452 \\ 
 a_{13}& & & & &  40.1942 &  15.8620 &   8.5900 &   5.4091 \\ 
 a_{14}& & & & & &  47.6992 &  18.6566 &  10.0367 \\ 
 a_{15}& & & & & & &  56.0662 &  21.7529 \\ 
 a_{16}& & & & & & & &  65.3412 \\ 
\hline 
 V_n(P) &   0.3447 &   0.2973 &   0.2593 &   0.2281 &   0.2023 &   0.1807 &   0.1624 &   0.1468 \\ 
\hline 
 \end{array} 

$$
\caption{Principal points for $n = 1$ through $n = 16$ for the t-distribution $f(x) = \Gamma\left({k+1 \over 2}\right) (\sqrt{k \pi} \Gamma\left(k / 2 \right))^{-1}
\left( 1 + x^2 / k \right)^{-\left( k+1 \right)/2}$
on the domain $(-\infty, \, \infty)$. For this table, the parameter $k = 3$. Note that we require $k \ge 3$ for the probability, expected value, and variance formula to be well defined. }
\label{tab.t}
\end{table}



\clearpage

\end{document}